\documentclass[12pt]{article}
\usepackage{hyperref,geometry,graphicx,amsmath,amssymb,amsfonts,amsthm,array}
\usepackage{pdflscape}

\newtheorem{thm}{Theorem}

\newtheorem{cor}[thm]{Corollary}
\newtheorem{defn}{Definition}
\newtheorem{exam}{Example}
\newtheorem{re}{Remark}
\newtheorem{dt}{Definition/Theorem}

\def\Res{{\rm Res}}

\newcommand{\CC}{\mathbb{C}} 
 \def\RR{{\mathbb{R}}} \newcommand{\ZZ}{\mathbb{Z}}

\def\Res{{\rm Res}}

\begin{document}
\begin{titlepage}
\title{Plane mixed discriminants and toric jacobians}
\author{Alicia Dickenstein \footnote{Email: alidick@dm.uba.ar,
Department of Mathematics, FCEN, University of Buenos Aires, and
IMAS, CONICET, Argentina}, 
Ioannis~Z.~Emiris \footnote{Email: emiris@di.uoa.gr, 
Department of Informatics \& Telecommunications,
University of Athens, Greece} 
and Anna Karasoulou \footnote{Email: akarasou@di.uoa.gr, 
Department of Informatics \& Telecommunications,
University of Athens, Greece}}

\end{titlepage}
\maketitle

\rightline{\em Dedicated to the memory of our friend Andrei Zelevinsky
(1953--2013)}

\bigskip

\abstract{Polynomial algebra offers a standard approach to handle several
problems in geometric modeling.
A key tool is the discriminant of a univariate polynomial, or of a
well-constrained system of polynomial equations,
which expresses the existence of a multiple root.
We describe discriminants in a general context, 
and focus on exploiting the sparseness of polynomials via
the theory of Newton polytopes and sparse (or toric) elimination.
We concentrate on bivariate polynomials and establish an original formula
that relates the mixed discriminant of two bivariate Laurent polynomials
with fixed support, with the sparse resultant of these polynomials 
and their toric Jacobian.
This allows us to obtain a new proof for the bidegree of the mixed discriminant
as well as to establish
multipicativity formulas arising when one polynomial can be factored.}

\section{Introduction} \label{sec:intro}

Polynomial algebra offers a standard and powerful approach to handle several
problems in geometric modeling.
In particular, the study and solution of systems of polynomial equations
has been a major topic.
Discriminants provide a key tool when examining well-constrained systems,
including the case of one univariate polynomial.  
Their theoretical study is a thriving and fruitful domain today,
but they are also very useful in a variety of applications.

The best studied discriminant is
probably known since high school, where one studies the discriminant
of a quadratic polynomial $f(x)=ax^2+bx+c=0$ ($a \not=0$).
The polynomial $f$ has a double root if and only if its discriminant
$\Delta_2=b^2-4ac$ is equal to zero.
Equivalently, this can be defined as the condition for $f(x)$
and its derivative $f'(x)$ to have a common root:
\begin{equation}
\exists\, x\; :\; f(x)=ax^2+bx+c=f'(x)=2ax+b=0\; \Leftrightarrow \Delta_2=0.
\end{equation}
One can similarly consider the discriminant of a univariate polynomial
of any degree.  
If we wish to calculate the discriminant $\Delta_5(f)$ of a polynomial $f$ of degree five
in one variable, we consider the condition that both $f$
and its derivative vanish:
\[ 
\begin{array}{l l} 
   f(x)=ax^5+bx^4+cx^3+dx^2+ex+g=0,\\ 
   f'(x)=5ax^4+4bx^3+3cx^2+2dx+e=0 .
\end{array} 
\] 
In this case,
elimination theory reduces the computation of $\Delta_{5}$ to the computation 
of a $9\times 9$ Sylvester determinant, which equals $a \, \Delta_5(f)$.
If we develop this determinant, we find out that the number monomials
in the discriminant increases rapidly with the input degree: 
\[ 
\begin{array}[width=0.01\textwidth]{l l} 
\Delta_5=-2050a^2g^2bedc+356abed^2c^2g-80b^3ed^2cg+18dc^3b^2g\\ 
e-746agdcb^2e^2+144ab^2e^4c-6ab^2e^3d^2-192a^2be^4d-4d^2ac\\ 
^3e^2+144d^2a^2ce^3-4d^3b^3e^2-4c^3e^3b^2-80abe^3dc^2+18b^3e^3\\ 
dc+18d^3acbe^2+d^2c^2b^2e^2-27b^4e^4-128a^2e^4c^2+16ac^4e^3-27\\ 
a^2d^4e^2+256a^3e^5+3125a^4g^4+160a^2gbe^3c+560a^2gdc^2e^2+1020\\ 
a^2gbd^2e^2+160ag^2b^3ed+560ag^2d^2cb^2+1020ag^2b^2c^2e-192\\ 
b^4ecg^2+24ab^2ed^3g+24abe^2c^3g+144b^4e^2dg-6b^3e^2c^2g+14\\ 
4dc^2b^3g^2-630dac^3bg^2-630d^3a^2ceg-72d^4acbg-72dac^4e\\ 
g-4d^3c^2b^2g-1600ag^3cb^3-2500a^3g^3be-50a^2g^2b^2e^2-3750a^3\\ 
g^3dc+2000a^2g^3db^2+2000a^3g^2ce^2+825a^2g^2d^2c^2+2250a^2g^3b\\ 
c^2+2250a^3g^2ed^2-900a^2g^2bd^3-900a^2g^2c^3e-36agb^3e^3-1600\\ 
a^3ge^3d+16d^3ac^3g-128d^2b^4g^2+16d^4b^3g-27c^4b^2g^2+108ac^5\\ 
g^2+108a^2d^5g+256b^5g^3 .
\end{array} 
\] 	
In fact,
if we compute the resultant of $f$ and $x f'$ by means of the $10 \times 10$
Sylvester determinant, we find the more symmetric output:
$a\, g \, \Delta_5(f)$. 
This formula is very well known for univariate discriminants~\cite{gkz}, and
we generalize it in Theorem~\ref{th:main}. 
%\item 1 variable, degree 5: $\Delta$  has $59$ terms 

One univariate polynomial is the smallest well-constrained system.
We are concerned with multivariate systems of sparse polynomials,
in other words, polynomials with
fixed support, or set of nonzero terms.
{\em Sparse (or toric)} elimination theory
concerns the study of resultants and discriminants 
associated with toric varieties. This theory has its origin in 
the work of Gel'fand, Kapranov and Zelevinsky on multivariate hypergeometric 
functions.  Discriminants arise as singularities of such functions \cite{gkz89}.

Gel'fand, Kapranov and Zelevinsky~\cite{gkz}
established a general definition of sparse discriminant, which gives as special case
the following definition of (sparse) mixed discriminant (see Section~\ref{sec:prelim} 
for the relation with the discriminant of the associated Cayley matrix 
and with the notion of mixed discriminant in~\cite{MD}).  In case $n=2$, the mixed discriminant 
detects tangencies between families of curves with fixed supports. In general,
the \textit{mixed discriminant} $\Delta_{A_1, \dots, A_n}(f_{1}, \dots ,f_{n})$
of $n$ polynomials in $n$ variables with fixed supports $A_1, \dots, A_n \subset \mathbb{Z}^n$
is the irreducible polynomial (with integer coprime coefficients, defined up to sign)  
in the coefficients of the $f_i$ which
vanishes whenever the system $f_1=\cdots=f_n=0$ has a multiple root (that is, a root which is not
simple) with non-zero coordinates, in case this discriminantal variety is a hypersurface (and equal to
the constant $1$ otherwise). The zero locus of 
the mixed discriminant is the variety of ill-posed systems \cite{SS}. 
We shall work with the polynomial 
defining the \emph{discriminant cycle} (see Section~\ref{sec:prelim}) 
which is defined as the power 
$\Delta_{A_1, \dots, A_n}^{i(A_1, \dots, A_n)}$ of the mixed discriminant raised to
the index 
\begin{equation}
 \label{eq:index}
i(A_1, \dots, A_n) = [\mathbb{Z}^n:\mathbb{Z} A_1 + \dots +\mathbb{Z} A_n],
\end{equation}
which stands for the index of 
lattice $\mathbb{Z} A_1+ \dots + \mathbb{Z} A_{n}$ in $\mathbb{Z}^n$.
In general, this index equals $1$ and so both concepts coincide.

Discriminants have many applications.
Besides the classical application in the realm of
differential equations to describe singularities, 
discriminants occur for instance
in the description of the topology of real algebraic plane curves \cite{GN},
in solving systems of polynomial
inequalities and zero-dimensional systems \cite{FMRS}, 
in determining the number of real roots of square systems of sparse polynomials \cite{DRRS},
in studying the stability of numerical solving \cite{D}, 
in the computation of the Voronoi diagram of curved objects \cite{ETT08}, or
in the determination of cusp points of parallel manipulators \cite{MRCW}.

Computing (mixed) discriminants is a (difficult) elimination problem. 
In principle, they can be computed with Gr\"obner bases, but this is
very inefficient in general since these polynomials have a rich combinatorial
structure \cite{gkz}. Ad-hoc computations via complexes (i.e., via tailored
homological algebra)
are also possible, but they also turn out to be complicated. The
tropical approach to compute discriminants was initiated in~\cite{DFS} and the
tropicalization of mixed planar discriminants was described in~\cite{DT}. 
Recently, in \cite{EKKL}, the authors focus on computing the discriminant of a multivariate
polynomial via interpolation, based on \cite{EFKP, Rin}; the latter essentially 
offers an algorithm for predicting the discriminant's Newton polytope,
hence its nonzero terms.
This yields a new output-sensitive algorithm which, however,
remains to be juxtaposed in practice to earlier approaches.

We mainly work in the case $n=2$, where the results are more transparent and the
basic ideas are already present, 
but all our results and methods can be generalized
to any number of variables. This will be addressed in a subsequent paper \cite{D13}. 
Consider for instance a system of two polynomials in two variables and
assume that, the first polynomial factors as $f_{1}=f'_{1}\cdot f''_{1}$. 
Then, the discriminant also factors and we thus
obtain a multiplicativity formula for it, which we make precise in 
Corollary~\ref{cor:factor}.
This significantly simplifies the discriminant's computation and generalizes the formula
in \cite{BJ} for the classical homogeneous case. This multiplicativity formula
is a consequence of our main result (Theorem~\ref{th:main} in dimension $2$, 
see also Theorem~\ref{th:BJ} in any dimension) 
relating the mixed discriminant and the resultant of
the given polynomials and their \emph{toric Jacobian} 
(see Section~\ref{sec:gral} for precise definitions
and statements). As another consequence of Theorem~\ref{th:main}, 
we reprove, in Corollary~\ref{cor:main2},
the bidegree formula for planar mixed discriminants in~\cite{MD}.

The rest of this chapter is organized as follows.
The next section overviews relevant existing work and definitions.
In Section~\ref{sec:gral} we present our main results relating the
mixed discriminant with the sparse resultant of the two polynomials
and their {toric Jacobian}.
In Section~\ref{sec:mult} we deduce the general multiplicativity 
formula for the mixed discriminant when one polynomial factors.

\section{Previous work and notation} \label{sec:prelim}

In this section we give a general description of discriminants and some
definitions and notations that we are going to use in the following sections.

Given a set $A \subset \mathbb{R}^n$, let $Q=conv(A)$ denote
the convex hull of $A$.
We say that $A$ is a lattice set or configuration if it is contained
in $\mathbb{Z}^n$, whereas
a polytope with integer vertices is called a lattice polytope.
We denote by $\mbox{Vol}(\cdot)$ the volume of a lattice polytope,
normalized with respect to the lattice ${\mathbb Z}^n$, so that
a primitive simplex has normalized volume equal to~$1$.
Normalized volume is obtained by multiplying Euclidean volume 
by $n!$.

Given a non-zero Laurent polynomial 
$$
f=\sum_{a}{c_{a}x^{a}}, 
$$
the finite subset $A$ of $\mathbb{Z}^{n}$ of those exponents $a$ for
which $c_a\not=0$ is called the \emph{support} of $f$. 
The \textit{Newton polytope} $N(f)$ of $f$
is the lattice polytope defined as the convex hull of $A$.

A (finite) set $A$ is said to be \textit{full}, if it consists of 
all the lattice points in its convex hull.
In~\cite{MD}, $A$ is called \emph{dense} in this case, but we prefer to reserve
the word dense to refer to the classical homogeneous case.
A subset $F \subseteq A$ is called a \textit{face} of $A$, denoted $F \prec A$, if $F$ 
is the intersection of $A$ with a face of the polytope $conv(A)$.

As usual $Q_{1}+Q_{2}$ denotes the Minkowski sum of sets
$Q_{1}$ and $Q_{2}$ in $\mathbb{R}^n$.
The \textit{mixed volume} $MV(Q_{1},\dots,Q_{n})$ of $n$ convex polytopes
$Q_{i}$ in $\mathbb{R}^{n}$ is the multilinear function with respect to
Minkowski sum that generalizes the notion of volume in the sense that 
$MV(Q, \dots, Q) = \mbox{Vol}(Q)$, when all $Q_{i}$ equal a fixed
convex polytope $Q$.

The following key result is due to Bernstein and Kouchnirenko.
The mixed volume of the Newton polytopes of $n$ Laurent polynomials
$f_{1}(x), \dots , f_{n}(x)$ in $n$ variables is an integer that
bounds the number of isolated common solutions of $f_{1}(x)=0, \dots , f_{n}(x)=0$
in the algebraic torus $(K^{*})^{n}$, over an algebraically closed field $K$ 
containing the coefficients. If the coefficients of the polynomials 
are generic, then the
common solutions are isolated and their number equals the mixed volume.
This bound generalized B\'ezout's classical bound to the sparse case:
for homogeneous polynomials
the mixed volume and B\'ezout's bound coincide. 
 
Mixed volume can be defined in terms of Minkowski sum volumes as follows.
\[
MV_{n}(Q_{1},\dots,Q_{n})= \sum_{k=1}^{n}(-1)^{n-k}
	\sum_{I\subset \{1,\dots,n\}, |I|=k}
	\frac{1}{n!} \mbox{Vol}\Big(\sum_{i\in I} Q_{i} \Big).
\] 
This implies, for $n=2$:
$$ 
2 MV(Q_{1},Q_{2})=
	\mbox{Vol}(Q_{1}+Q_{2})-\mbox{Vol}(Q_{1})-\mbox{Vol}(Q_{2}).
$$

\begin{def}
\label{def:essential}
A family of finite lattice configurations $A_1, \dots, A_{k}$ in $\mathbb{Z}^n$
is called essential if the affine dimension of the lattice $\mathbb{Z} A_1 + \dots +\mathbb{Z} A_k$
equals $k-1$, and for all proper subsets $I \subset \{1, \dots, k\}$ it holds that
the affine dimension of the lattice generated by $\{A_i, i \in I\}$ is greater or
equal than its cardinality $|I|$.
\end{def}

\begin{dt}\label{resultant}{\rm\cite{gkz,S}}
Fix a family of $n+1$ finite lattice configurations $A_1, \dots, A_{n+1}$ which contains a
unique essential subfamily $\{A_i, i \in I\}$. Given Laurent polynomials in $n$ variables
$f_{1}, \dots , f_{n+1}$  with supports $A_1, \dots, A_{n+1}$,
the \textit{resultant} $\Res_{A_1, \dots, A_{n+1}} (f_{1}, \dots, f_{n+1})$ 
is the irreducible polynomial with coprime integer coefficients (defined up to sign)
in the coefficients of $f_{1}, \dots, f_{n+1}$, which vanishes whenever
$f_{1}, \dots , f_{n+1}$ have a common root in the torus $(\mathbb{C}^*)^n$. In fact, in
this case, the resultant
%$\Res_{A_1, \dots, A_{n+1}} (f_{1}, \dots, f_{n+1})$ 
only depends on the coefficients of $f_i$ with $i \in I$.

If there exist more than one essential subfamilies, then the (closure of the) 
variety of solvable systems
is not a hypersurface and in this case we set: 
$$\Res_{A_1, \dots, A_{n+1}}(f_{1}, \dots, f_{n+1})= 1.$$
\end{dt}

In what follows, we consider  $n$ (finite) lattice configurations 
$A_{1},\dots,A_{n}$  in $\mathbb{Z}^{n}$ and we denote by $Q_1, \dots, Q_n$
their respective convex hulls. Let $f_1, \dots, f_n$ be Laurent polynomials
with support $A_1, \dots, A_n$ respectively:
\[ f_i(x) = \sum_{\alpha \in A_i} c_{i, \alpha} x^\alpha, \quad i = 1 \dots, n.\]
In \cite{MD} the \emph{mixed discriminantal variety}, is defined as closure of the locus of coefficients 
$c_{i,\alpha}$ for which the associated system $f_1= \dots = f_n=0$  has a non-degenerate multiple root
$x \in (K^{*})^{n}$. This means that $x$ is an isolated root and
the $n$ gradient vectors 
\[\left(\frac{\partial f_{i}}{\partial x_{1}}(x), \dots , \frac{\partial f_{i}}{\partial x_{n}}(x)\right)\]
are linearly dependent, but any $n-1$ of them are linearly independent.

\begin{def}\label{def:md}
If the mixed discriminantal variety is a hypersurface,  the \textit{mixed discriminant} 
of the previous system is the unique up to sign irreducible polynomial 
$\Delta_{A_{1}, \dots, A_{n}}$ with integer coefficients in the unknowns 
$c_{i,a}$ which defines this hypersurface.
Otherwise, the family is said to be defective and we set $\Delta_{A_{1}, \dots, A_{n}}=1$.
The \emph{ mixed discriminant cycle} $\tilde{\Delta}_{A_{1},\dots,A_{n}}$ is equal to
$i(A_1, \dots, A_n)$ times the mixed discriminant variety, and thus
its equation equals $\Delta_{A_{1},\dots,A_{n}}$ raised to this integer (defined in \eqref{eq:index}).
\end{def}

By \cite[Theorem~2.1]{MD}, when the family $A_1, \dots, A_n$ is non defective,
the mixed discriminant $\Delta_{A_{1},\dots,A_{n}}$ 
coincides with the $A$-discriminant defined in \cite{gkz}, where $A$ is the Cayley matrix
\[A=\left(
\begin{array}{cccc}
	   1&      0&    \dots    &0\\
   	 0&      1&    \dots    &0\\
\dots & \dots &    \dots    &\dots\\
   	 0&      0&    \dots    &1\\
 A_{1}&  A_{2}&    \dots    &A_{n}
\end{array}\right).\] 
This matrix has $2n$ rows and $m=\sum_{ i=1}^{n}|A_{i}|$ columns, 
so $0=(0,\dots,0)$ 
and $1=(1,\dots,1)$ denote row vectors of appropriate lengths.
We introduce $n$ new variables $y_{1},\dots,y_{n}$  
in order to encode the system 
$f_{1}=\dots=f_{n}=0$ in one polynomial with support in $A$, 
via the \emph{Cayley trick}: 
$\phi(x,y)=y_{1}f_{1}(x)+\dots+y_{n}f_{n}(x)$. 
Note that $i(A_1, \dots, A_n) = [\mathbb{Z}^{2n}, \mathbb{Z} A]$.

In what follows when we refer to resultants or discriminants we will 
refer to the equations of the corresponding cycles, but we will omit
the tildes in our notation. More explicitly, we will follow the convention 
in the article \cite{DS} by D'Andrea and Sombra. 
In general, both definitions coincide, but this convention 
allows us to present cleaner formulas. For instance, when the family 
$A_1, \dots, A_{n+1}$ is essential, our notion
of resultant equals the resultant in \cite{gkz,S} raised to the index
$i(A_1, \dots, A_{n+1})$. 
In most examples these two lattices coincide, and so our resultant cycle 
equals the resultant variety and the associated resultant polynomial is
irreducible.
 
\begin{re}\label{re:monomial}
Assume $A_1$ consists of a single point $\alpha$ and that $\{1\}$ is the 
only essential subfamily of a given family
$A_1, \dots, A_{n+1}$. Let $f_1(x)= c x^{\alpha}$. Then, for any choice of 
Laurent polynomials $f_2, \dots, f_{n+1}$
with supports $A_2, \dots, A_{n+1}$, it holds that (cf. \cite[Proposition~2.2]{DS})
\begin{equation}\label{eq:expc}
{\rm Res}_{A_1, \dots, A_{n+1}}(f_1, \dots, f_n) \, = \, c^{MV(A_2, \dots,A_{n+1})}.
\end{equation}
\end{re}

With this convention, the following multiplicativity formula holds:

\begin{thm}{\rm\cite{DS,PS}}\label{th:productres}
Let $A'_1,A''_1, A_1, \dots, A_{n+1}$ be finite subsets of $\ZZ^n$ 
with $A_1 = A'_1+A''_1$. Let $f_1, \dots, f_{n+1}$ be
polynomials with supports contained in $A_1,\dots, A_{n+1}$ and assume 
that $f_1 = f'_1 f''_1$ where $f'_1$ has support $A'_1$ and
$f''_1$ has support $A''_1$. Then
\begin{equation*}
{\rm Res}_{A_1, \dots, A_{n+1}}(f_1, \dots, f_{n+1}) \, = \, 
{\rm Res}_{A'_1, \dots, A_{n+1}}(f'_1, \dots, f_{n+1}) \cdot 
{\rm Res}_{A''_1, \dots, A_{n+1}}(f''_1, \dots, f_{n+1}).
\end{equation*}
\end{thm}

Cattani, Cueto, Dickenstein, Di Rocco and Sturmfels in \cite{MD} proved that 
the degree of the mixed discriminant $\Delta$ is a piecewise linear function in the 
Pl\"ucker coordinates of a mixed Grassmanian. 
An explicit degree formula for plane curves is also presented in \cite[Corollary~3.15]{MD}. 
In case $A_1, A_2$ consist of all the lattice points in their convex hulls, 
they are two dimensional and with the same normal fan, then
the bidegree of $\Delta_{A_{1},A_{2}}$ satisfies the following:
bidegree of $\Delta_{A_{1},A_{2}}$ in the coefficients of $f_i$ equals:
\begin{equation*}
= {\rm Vol}(Q_{1}+Q_{2})- area(Q_{i}) -perimeter(Q_{j}),
\end{equation*}
 where $i\in\{1,2\}, \, i \ne j$.
where $Q_{i}=conv(A_{i})$, $i=1,2$, and $Q_{1}+Q_{2}$ is their Minkowski sum. 
The area is normalized, so that a primitive triangle has area $1$ and the 
perimeter of $Q_{i}$ is the cardinality of $\partial Q_{i}\cap \mathbb{Z}^{2}.$
We will recover the general formula for this degree and present it in 
Corollary~\ref{cor:main2}.

Bus\'e and Jouanolou consider in~\cite{BJ} the following equivalent 
definition of the mixed discriminant, in case where
$f_{1},\dots,f_{n}$ are dense homogeneous polynomials in 
$(x_0,\dots,x_n)$ of degrees $d_{1},\dots,d_{n}$ respectively, 
that is, their respective supports 
$A_i = d_i \sigma $ are all
the lattice points in the $d_i$-th dilate  of the unit simplex $\sigma$ in $\RR^n$. 
It is the non-zero polynomial in the coefficients of $f_1,\dots, f_n$ which equals
\begin{equation} \label{eq:BJ}
\frac{\Res_{d_1 \sigma, \dots, d_n \sigma, \delta_i \sigma}
(f_{1},\dots,f_{n},J_{i})}{\Res_{d_1 \sigma, \dots, d_n \sigma, \sigma}(f_{1},\dots,f_{n},x_{i})},
\end{equation}
for all $i\in \{1,\dots,n\}$, where $J_i$ is the maximal minor of the 
Jacobian matrix associated to $f_1, \dots, f_{n}$ obtained by deleting the
$i$-th.  
We give a more symmetric and general formula in Corollary~\ref{cor:BJ} below.

The multiplicativity property of the discriminant in the case 
of dense homogeneous polynomials
was already known to Sylvester \cite{Syl} and  generalized by Bus{\'{e}} 
and Jouanolou in \cite{BJ}, where they proved that when in particular 
$A_1 = d_1 \sigma = (d'_1+ d''_1) \sigma$ 
and $f_1$ is equal to the product of two polynomials 
$f'_1 \cdot f''_1$ with respective degrees $d'_1, d''_1$, 
the following factorization holds:
\begin{equation}\label{eq:guess}\begin{split}
\Delta_{d_{1}\sigma, \dots, d_n \sigma}(f_{1}, \dots, f_{n}) \, = \, &
\Delta_{d'_{1}\sigma, \dots ,d_{n}\sigma}(f'_{1}, \dots ,f_{n})
\cdot \Delta_{d''_{1}\sigma, \dots ,d_{n}\sigma}(f''_{1}, \dots ,f_{n}) \\
& \cdot {\rm Res}_{d'_{1}\sigma, d''_1 \sigma, \dots ,d_{n}\sigma}(f'_{1},f''_{1},
\dots ,f_{n})^{2}.\end{split}
\end{equation}

It is straightforward to see in general from the definition, that the
vanishing of any of the polynomials
 $\Delta_{A'_{1}, \dots ,A_{n}}(f'_{1}, \dots ,f_{n})$, $\Delta_{A''_{1}, 
\dots ,A_{n}}(f''_{1}, \dots ,f_{n})$, or ${\rm Res}_{A'_{1},A''_{1},
 \dots ,A_{n}}(f'_{1},f''_{1}, \dots ,f_{n})$ 
implies that \[\Delta_{A'_{1}+A''_{1}, \dots ,A_{n}}(f'_{1}f''_{1}, f_2, \dots ,f_{n})=0.\]

It follows from \cite{E10} that when each support configuration $A_i$ is  full,
the Newton polytope of 	the discriminant $\Delta_{A'_{1}+A''_{1}, A_{2},\dots, A_n}
(f'_{1}f''_{1}, f_2, \dots, f_n)$ 
equals the Minkowski sum of the Newton polytopes of the discriminants 
$\Delta_{A'_{1},A_{2}, \dots, A_n}(f'_{1}, f_{2}, \dots, f_n)$ and $\Delta_{A''_{1},A_{2}, \dots, A_n}
(f''_{1}, f_{2}, \dots, f_n)$ plus
two times the Newton polytope of the resultant
${\rm Res}_{A'_{1},A''_{1}, A_{2}, \dots, A_n}(f'_{1},f''_{1}, f_{2}, \dots, f_n)$.
So, a first guess would be that the factorization into the three factors 
in~\eqref{eq:guess} above holds for general supports. We will see in 
Corollary~\ref{cor:factor} that indeed other factors may occur, which we describe explicitly.

This behaviour already occurs in the univariate case:

\begin{exam} \label{ex:univariate}
Let $A'_{1} = \{0, i_{1}, \dots  , i_{m}, d_{1}\},
A''_{1} = \{0, j_{1}, \dots , j_{l}, d_{2}\}$ be the support sets of  
$f'_{1} = a_{0}+a_{i_{1}}x^{i_{1}}+ \dots +a_{i_{m}}x^{i_{m}}+a_{d_{1}}x^{d_{1}}, f''_{1} = 
b_{0}+b_{j_{1}}x^{j_{1}}+\dots +b_{j_{l}}x^{j_{l}}+b_{d_{2}}x^{d_{2}}
$ respectively. Then 
\[
\Delta(f'_{1}f''_{1})=\Delta(f'_{1})\cdot \Delta(f''_{1})\cdot R(f'_{1},f''_{1})^{2}\cdot E,
\]
where
$E =a_{0}^{i_{1}-m_{0}} \, b_{0}^{j_{1}-m_{0}}\, a_{d_{1}}^{d_{1}-i_{m}-m{1}} \, b_{d_{2}}^{d_{2}-j_{l}-m_{1}},$ 
with $m_{0} := \min\{i_{1}, j_{1}\}$ and $m_{1} := \min\{d_{1}-i_{m}, d_{2}-j_{l}\}.$
On the other hand, in the full case
$i_{1} = j_{1} = 1, i_{m} = d_{1}-1, j_{l} = d_{2}-1$, 
thus $E = 1$ because its exponents are equal to zero.
\end{exam}

\section{A general formula}\label{sec:gral}

The aim of this section is to present a formula which relates the mixed 
discriminant with the resultant of the given polynomials and their 
toric Jacobian, whose definition we recall.

\begin{defn} \label{def:JacT}
Let $f_1(x_1,\ldots,x_n),\ldots,f_n(x_1,\ldots,x_n)$ be $n$ Laurent polynomials in $n$ variables. 
The associated toric Jacobian $J_f^T$ equals $x_1 \cdots x_n$ times the 
determinant of the \textit{Jacobian matrix of f}, 
or equivalently, the determinant of the matrix:
$$\begin{bmatrix} x_1\dfrac{\partial f_1}{\partial x_1} & \cdots & 
x_n\dfrac{\partial f_1}{\partial x_n} \\ \vdots & \ddots & \vdots \\ 
x_1\dfrac{\partial f_n}{\partial x_1} & \cdots & x_n\dfrac{\partial f_n}
{\partial x_n} \end{bmatrix}.$$
\end{defn}
Note that the Newton polytope of $J_f^T$ is contained in the sum of the 
Newton polytopes of $f_1, \dots, f_n$.

As we remarked before, we will mainly deal in this chapter with the case $n=2$.
Also, to avoid excessive notations and make the main results cleaner,
we assume below that $A_1, A_2$ are two finite lattice configurations
whose convex hulls satisfy
$$
\dim(Q_1) = \dim(Q_2)=2.
$$ 

Let $f_1, f_2$ be polynomials with respective supports $A_1, A_2$:
\[ f_i(x) = \sum_{\alpha \in A_i} c_{i, \alpha} x^\alpha, \quad i = 1,2,\]
where $x=(x_1, x_2)$.
We denote
by $\Sigma$ the set of primitive inner normals $\eta\in (\ZZ^2)^*$ of the edges 
of $A_1+A_2$. We call $A_i^\eta$ the face of $A_i$ where the inner 
product with $\eta$ is minimized. 
We call this minumum value $\nu_i^\eta$. We also denote
by $f_i^\eta$ the subsum of terms in $f_i$ with exponents in this face
\[ f_i^\eta (x) = \sum_{\alpha \in A_i^\eta} c_{i, \alpha} x^\alpha, \quad i = 1,2,\]
which is $\eta$-homogeneous of degree $\nu_i^\eta$. Up to
multiplying $f_i$ by a monomial (that is, after translation of $A_i$) we can assume
without loss of generality that $\nu_i^\eta \not=0$.
Now, $A_i^\eta$ is either a vertex of $A_i$ (but not of both $A_1, A_2$
since two vertices do not give a Minkowski sum edge),
or its convex hull is an edge of $A_i$ (with inner normal $\eta$),
which we denote by $e_i^\eta$. 
 Note that if the face of $A_1+A_2$ associated to $\eta$ is a vertex,
both polynomials $f_i^\eta$ are monomials and their resultant
locus has codimension two. 
 
We denote by $\mu_i(\eta)$ $(i=1,2)$ the integer defined by
the following difference:
\begin{equation}\label{eq:mui}
\mu_i(\eta) = {\rm min} \{ \langle \eta, m \rangle, m \in A_i - A_i^\eta \} \, - \, \nu_i^\eta.
\end{equation}
and by 
\begin{equation}\label{eq:mu}
\mu(\eta) = {\rm min} \{ \mu_1(\eta), \mu_2(\eta) \},
\end{equation} 
the minimum of these two integers. Note that by our assumption that $\dim(Q_i) = 2$,
we have that $\mu(\eta) \ge 1$.

Without loss of generality, we can translate the support sets $A_1^\eta, A_2^\eta$ 
to the origin and consider the line $L^\eta$ containing them.
The residue (cycle) $\Res_{A_1^\eta, A_2^\eta}(f_1^\eta, f_2^\eta)$ 
is considered as before, with respect to the lattice $L^\eta \cap \ZZ^2$.

\begin{re}\label{rem:2}
As in Remark~\ref{re:monomial}, if $f_1^\eta$ is a monomial, the resultant 
equals the coefficient of $f_1^\eta$ raised to the normalized length 
$\ell(e_2^\eta)$ of the edge $e_2^\eta$ of $A_2$ 
 (that is, the number of integer points in the edge, minus $1$).
If $\eta$ is an inner normal of edges $A_1^\eta$ and $A_2^\eta$,
then the resultant equals the irreducible resultant raised to the index
of $\ZZ A_1^\eta + \ZZ A_2^\eta$ in $L^\eta \cap \ZZ^2$. In particular, 
the exponent $\mu(\eta) = 1$ if at least one of the
configurations is full. 
\end{re}

The following is our main result. 

\begin{thm} \label{th:main}
Let $f_1, f_2$ be generic Laurent polynomials with respective supports
$A_1, A_2$. Then,
\[
{\rm Res}_{A_1,  A_2, A_1+ A_2}(f_1,  f_2, J_f^T) = \Delta_{A_{1}, A_{2}}(f_{1}, f_{2}) \cdot E,
\]
where the factor $E$ equals the finite product:
\[E  =  \prod_{\eta \in \Sigma} \Res_{A_1^\eta, A_2^\eta}(f_1^\eta, f_2^\eta)^{\mu(\eta)}.\]
\end{thm}

\begin{proof}
Let $X$ be the projective toric variety associated to $A_1+A_2$.  This compact 
variety consists of an open dense set $T_X$ isomorphic to the torus $(\CC^*)^2$ plus one toric divisor
$D_\eta$ for each $\eta \in \Sigma$. The Laurent polynomials $f_1, f_2, J^T_f$
define sections $L_1, L_2, L_J$ of globally generated line bundles on $X$.
The resultant ${\rm Res}_{A_1,  A_2, A_1+ A_2}(f_1,  f_2, J_f^T)$
vanishes if and only if $L_1, L_2, L_J$ have a common zero on $X$, which could be
at $T_X$  or at any of the $D_\eta$.

There is an intersection point at $T_X$ if and only if
there is a common zero of $f_1, f_2$ and $J^T_f$ in the torus $(\CC^*)^2$. In this case, the discriminant
$\Delta_{A_{1}, A_{2}}(f_{1}, f_{2})$ would vanish.  It follows that $\Delta_{A_{1}, A_{2}}(f_{1}, f_{2})$
divides ${\rm Res}_{A_1,  A_2, A_1+ A_2}(f_1,  f_2, J_f^T)$. 
(the indices $[\ZZ^2:\ZZ A_1 +  \ZZ A_2]$
and $[\ZZ^2:\ZZ A_1 + \ZZ A_2 + \ZZ (A_1 + A_2)]$ are equal).

If instead there is a common zero at some $D_\eta$,
this translates into the fact that $f_1^\eta, f_2^\eta$ and $(J_f^T)^\eta = J^T_{f^\eta}$ 
(with obvious definition) have a common solution.
But as $f_i^\eta$ are $\eta$-homogeneous, they satisfy the weighted Euler equalities:
\begin{equation}
\eta_1 x_1\dfrac{\partial f_i^\eta}{\partial x_1} + \eta_2 x_2\dfrac{\partial f_i^\eta}{\partial x_2} = \nu_i^\eta
f_i, \quad i=1,2,
\end{equation}
 from which we deduce that $J^T_{f^\eta}$ lies in the ideal $I(f_1^\eta, f_2^\eta)$ and so,
 the three polynomials will vanish exactly when there is a nontrivial
common zero of $f_1^\eta$ and $f_2^\eta$.
 This implies that all facet resultants $\Res_{A_1^\eta, A_2^\eta}(f_1^\eta, f_2^\eta)$ divide
${\rm Res}_{A_1,  A_2, A_1+ A_2}(f_1,  f_2, J_f^T)$. 

%Let $x=(x_1, x_2)$ and \[ f_i(x) = \sum_{alpha \in A_i} c_{i, \alpha} \]
Now, we wish to see that the resultant
$\Res_{A_1^\eta, A_2^\eta}(f_1^\eta, f_2^\eta)$ raised to the power
$\mu(\eta)$ occurs as a factor.
The following argument would be better written in terms of 
the multihomogeneous polynomials in the Cox coordinates of $X$ 
which represent $L_1, L_2, L_J$ \cite{CDS}. 
Fix a primitive inner normal direction $\eta \in \Sigma$ of $A_1+A_2$, 
let $t$ be a new variable and define the following polynomials
\begin{equation}\label{Fi}
F_i(t,x) = \sum_{\alpha \in A_i} c_{i, \alpha} t^{\langle \eta, 
\alpha \rangle - \nu_i^\eta} x^\alpha, \quad i = 1,2,
\end{equation}
so that
\[ F_i(1,x) = f_i (x), \quad F_i(0,x) = f_i^\eta(x), \quad i=1,2,\]
and we can write
\begin{equation}\label{eq:fieta}
f_i^\eta(x) = F_i^\eta(t,x) - t^{\mu_i(\eta)} G_i(t,x), \quad i=1,2,
\end{equation}
where the polynomials $G_i$ are defined by these equalities.
The polynomials $F_1, F_2, J^T_F$ define the sections $L_1, L_2, L_J$.
For each $t$, we deduce from the bilinearity of the determinant, 
that there exists a polynomial $H(t,x)$ such that
the toric Jacobian can be written as $J^T_F = J^T_{f^\eta} + t^{\mu(\eta)} H(t,x)$.
But, as we remarked, $J^T_{f^\eta}$ lies in the ideal $I(f_1^\eta, f_2^\eta)$, and using the equalities~\eqref{eq:fieta}, 
we can write $J^T_F = H_1(t,x) + t^{\mu(\eta)} H_2(t,x)$, with $H_1 \in I(F_1, F_2)$.
Note that if for instance $\eta_1 \not=0$, then the power of $x_1$ in each monomial
of $F_i$ can be obtained from the power of $t$ and the power of $x_2$, that is, we
could use $t$ and $x_2$ as ``variables'' instead.  
We will denote by ${\rm Res}^X$ the resultant
defined over $X$ \cite{CDS}.
Therefore,
\[{\rm Res}_{A_1,  A_2, A_1+ A_2}(f_1,  f_2, J_f^T) = 
{\rm Res}^X_{A_1,  A_2, A_1+ A_2}(F_1,  F_2, t^{\mu(\eta)} H_2).\]
Now, it follows from Theorem~\ref{th:productres} that
\[{\rm Res}^X_{A_1,  A_2, A_1+ A_2}(F_1,  F_2, t^{\mu(\eta)}) = \Res_{A_1^\eta, A_2^\eta}(f_1^\eta, f_2^\eta)^{\mu(\eta)}\]
is a factor of ${\rm Res}_{A_1,  A_2, A_1+ A_2}(f_1,  f_2, J_f^T)$. 
Indeed, no positive power of $t$ divides $H_2$ for generic coefficients. Considering all
possible $\eta \in \Sigma$ we get the desired factorization.
\end{proof}

Theorem \ref{th:main} and the proof will be extended to the general
$n$-variate setting in a forthcoming paper \cite{D13}.
We only state here the following general version without proof. Recall that a
lattice polytope $P$ of dimension $n$ in $\RR^n$ 
is said to be \emph{smooth} if at each every vertex there
are $n$ concurrent facets and their primitive inner normal directions
form a basis of $\ZZ^n$. In particular, integer dilates of the unit simplex or
the unit (hyper)cube are smooth.

\begin{thm}\label{th:BJ}
Let $P \subset \RR^n$ be a smooth lattice polytope of dimension $n$. 
Let $A_i = (d_i P) \cap \ZZ^n$, $i=1, \dots, n$, $d_1, \dots, d_n \in \ZZ_{>0}$, and $f_1, \dots, f_n$ 
polynomials with these supports, respectively. Then, we
have the following factorization
\[{\rm Res}_{A_1, \dots, A_n, A_1+ \dots+ A_n}(f_1, \dots, f_n, J_f^T) = \Delta_{A_{1},\dots, A_{n}}
(f_{1}, \dots, f_{n}) \cdot E,\]
where the factor $E$ equals the finite product:
\[E  =  \prod_{\eta \in \Sigma} \Res_{A_1^\eta, A_2^\eta}(f_1^\eta, f_2^\eta).\]
\end{thm}
Note that all the exponents in $E$ equal $1$ and all the lattice indices equal $1$.

When the given lattice configurations $A_i$ are the lattice points 
$d_i \sigma$ of the $d_i$-th dilate of the standard simplex
$\sigma$ in $\RR^n$, (that is, in the homogeneous case studied in~\cite{BJ}), 
formula~\eqref{eq:BJ} gives for any $n$ in our notation:
\begin{align*}
&{\Res_{d_1 \sigma, \dots, d_n \sigma, \delta \sigma}(f_{1},\dots,f_{n},J_{i})} = &\\
&\Delta_{d_{1} \sigma,\dots,d_{n} \sigma} (f_1, \dots, f_n)\cdot 
{\Res_{(d_1 \sigma)^{e_i}, \dots, (d_n \sigma)^{e_i}}
	(f_{1}^{e_i},\dots,f_{n}^{e_i})} ,&
\end{align*}
where $e_0, \dots, e_n$ are the canonical basis vectors
(or $e_0 = -e_1 -\dots - e_n$, if we consider
the corresponding dehomogenized polynomials, by setting $x_0=1$).
Note that Theorem~\ref{th:BJ} gives the following more symmetric formula:

\begin{cor}\label{cor:BJ}
With the previous notation, it holds:
\begin{align*}
&{\Res_{d_1 \sigma, \dots, d_n \sigma , (d_1 +\dots +d_n) \sigma}
(f_{1},\dots,f_{n},J_{f}^T)} =&\\ 
&\Delta_{d_{1} \sigma,\dots,d_{n}\sigma}(f_1, \dots, f_n) 
\cdot \prod_{i=0}^n{\Res_{(d_1 \sigma)^{e_i}, \dots, 
(d_n \sigma)^{e_i}}(f_{1}^{e_i},\dots,f_{n}^{e_i})}.&
\end{align*}
\end{cor}

It is straightforward to deduce from this expression the degree of 
the homogeneous mixed discriminant, obtained independently in
\cite{OB,BJ,N}.
Similar formulas can be obtained, for instance, in the multihomogeneous case.

\smallskip

We recall the following definition from~\cite{MD}.
If $v$ is a vertex of $A_{i}$, 
we define its \textit{mixed multiplicity} as
\begin{equation}
mm_{A_1,A_2}(v):=MV(Q_{1},Q_{2})-MV(C_{i},Q_{j}),\quad \{i,j\}=\{1,2\}, 
\end{equation}
where $C_{i}=conv(A_{i}-\{v\})$. 

Let $\Sigma'\subset \Sigma$ be the set of inner normals
of $A_1+A_2$ that cut out, or define, edges $e_i^{\eta}$ in both $Q_1, Q_2$.
The factorization formula in Theorem~\ref{th:main} can be written as follows,
and allows us to recover the
bidegree formulas for planar mixed discriminants in~\cite{MD}.  

\begin{cor}\label{cor:main2} 
Let $A_1, A_2$ be two lattice configurations of dimension $2$ in the plane, and let $f_1, f_2$ be
polynomials with these respective supports. 
Then, the resultant of $f_1,f_2$ and their toric Jacobian,
namely $\Res_{A_1, A_2,A_1+A_2} (f_1, f_2, J_f^T)$,
factors as follows:
\begin{equation}\label{eq:mm}
\Delta_{A_1, A_2}(f_1, f_2) \cdot 
\prod_{v \text{ vertex of } A_1 \text{ or } A_2} c_v^{\rm mm_{A_1, A_2}(v)}
\cdot
\prod_{\eta \in \Sigma'} {\Res_{A_1^\eta, A_2^\eta}(f_1^\eta, f_2^\eta)}^{\mu(\eta)}.
\end{equation}
The bidegree $(\delta_{1},\delta_{2})$ of the mixed discriminant
$\Delta_{A_1, A_2}(f_1, f_2)$ 
in the coefficients of $f_1$ and $f_2$, respectively, is then given
by the following:
%$\delta_{i}=deg_{A_{i}}(\tilde{\Delta}_{A_{1},A_{2}}) =$
\begin{equation}\label{eq:bidegree}
\mbox{Vol}(Q_{j})+2 \cdot MV(Q_{1},Q_{2})-\sum_{\eta \in \Sigma'} 
 \ell(e_j^\eta)  \cdot \mu(\eta)-
 \sum_{v \text{ vertex of } (A_i)}mm_{A_1, A_2}(v),
\end{equation}
 where $i\in\{1,2\}, \, i \ne j$.
\end{cor} 

\begin{proof}
To prove equality \eqref{eq:mm}, we need to show by Theorem \ref{th:main} that the factor 
\[ E = \prod_{\eta \in \Sigma} \Res_{A_1^\eta, A_2^\eta}(f_1^\eta, f_2^\eta)^{\mu(\eta)}
\]
equals the product
\[
\prod_{v \text{ vertex of } A_1 \text{ or } A_2} c_v^{\rm mm_{A_1, A_2}(v)}
\cdot
\prod_{\eta \in \Sigma'} {\Res_{A_1^\eta, A_2^\eta}(f_1^\eta, f_2^\eta)}^{\mu(\eta)}.
\]
When $\eta \in \Sigma'$, i.e.\ $\eta$ is a common
inner normal to edges of both $Q_i$, we get the same factor on both terms, since  
that our quantity $\mu(\eta)$ equals the index
$\min \{u(e_1(\eta), A_1), u(e_2(\eta), A_2)\}$, in the notation of~\cite{MD}.
%In this case, $\mu(\eta) = \min \{u(e_1(\eta), A_1), u(e_2(\eta), A_2) \}$ and 
%In the notation of \cite{MD}, we have that $(e_1^\eta,f_1^\eta) \in P$ and the bidegree
%of this resultant equals $(\ell(f_\eta), \ell(e_\eta))$. 

Assume then that $\eta$ is only an inner normal to $Q_2$.
So, $A_1^\eta$ is a vertex $v$, $f_1^\eta = c x^v$ 
is a monomial (with coefficient $c$) and $f_2^\eta$ is a polynomial 
whose support equals the edge $e_2^\eta$ of $A_2$ orthogonal to $\eta$. 
In this case, $\Res_{A_1^\eta,A_2^\eta}(f_1^\eta, f_2^\eta) = 
c^{\ell(f_\eta)}$ by Remark~\ref{re:monomial}.

For such a vertex $v$,  denote by ${\mathcal E}(v)$ the set of those $\eta' \notin \Sigma'$
for which $v + e_2^{\eta'}$ is an edge of $Q_1 + Q_2$.
Note that it follows from the proof of \cite[Prop.3.13]{MD} (cf 
in particular Figure~1 there), that there exist non negative
integers $\mu'(\eta')$ such that
\[
{\rm mm}(v) = \sum_{\eta' \in\ {\mathcal E}(v)} \ell(e_2^{\eta'}) \cdot \mu'(\eta').
\]
Indeed, $\mu(\eta') =\mu'(\eta')$.

To compute the bidegree, we use the multilinearity of the mixed volume
with respect to Minkowski sum. 
Observe that the toric Jacobian has bidegree $(1,1)$ in
the coefficients of $f_1, f_2$, from which we get that the bidegree of 
the resultant
$\Res_{A_1, A_2,A_1+A_2} (f_1, f_2, J_f^T)$ is equal to
\begin{equation}\label{eq:degres}
(2 MV(A_1, A_2) + \mbox{Vol}(Q_2), \; 2 MV(A_1, A_2) + \mbox{Vol}(Q_1)).
\end{equation}
Substracting the degree of the other factors and taking into account that
the bidegree of the resultant $\Res_{A_1^\eta, A_2^\eta}(f_1^\eta, f_2^\eta)$
equals $(\ell(e_2^\eta), \ell(e_1^\eta))$, we deduce
the formula~\eqref{eq:bidegree}, as desired.
\end{proof}

\section{The multiplicativity of the mixed discriminant}\label{sec:mult}

This section studies the factorization of the discriminant when one
of the polynomials factors.
We make the hypothesis that $f'_{1}, f''_{1}, f_2$ have fixed support sets,
and $A'_{1},A''_{1},A_{2}\subseteq \mathbb{Z}^{2}$. 
So $f_{1}=f'_{1}\cdot f''_{1}$ has support in the Minkowski sum
$A_1 :=  A'_{1} + A''_{1}$; in fact, its support is generically
equal to $A_1$. 
We will denote by ${\mu'}(\eta)$ (resp. ${\mu''}(\eta)$) the integer defined in \eqref{eq:mu},
with $A_1$ replaced by $A'_1$ (resp. $A''_1$). 

\begin{cor} \label{cor:factor} 
Assume $A'_{1}, A''_{1}$ and $A_{2}$ are full
planar configurations of dimension $2$. 
Let $f'_1, f''_1, f_2$ be generic polynomials with these supports and let $f_1 = f'_1 \cdot f''_1$.
Then,
\[\Delta_{A_{1},A_{2}}(f_{1},f_{2})=\Delta_{A^{'}_{1},A_{2}}(f^{'}_{1},f_{2})
\cdot \Delta_{A^{''}_{1},A_{2}}(f^{''}_{1},f_{2})\cdot 
\Res_{A^{'}_{1},A^{''}_{1},A_{2}}(f^{'}_{1},f^{''}_{1},f_{2})^{2} \cdot E,\] 
where $E$ equals the following product:
\begin{equation}
\prod_{\eta \in \Sigma} \Res_{(A'_1)^\eta, A_2^\eta}((f'_1)^\eta, f_2^\eta)^{{\mu'}(\eta)-\mu(\eta)}
\cdot \Res_{(A''_1)^\eta, A_2^\eta}((f''_1)^\eta, f_2^\eta)^{{\mu''}(\eta)-\mu(\eta)}.
\end{equation}
\end{cor}

\begin{proof} 

By Theorem~\ref{th:main}, we get that
%%%%%%%%%%%%%%%%%%%%%%%%%%%%%%%%%%%%%%%%%%%%%%%%%%%%%%%%%%%%%%%
\begin{equation}\label{deltafrac}
\Delta_{A_{1},A_{2}}(f_{1},f_{2})=\frac{R_{A_{1},A_{2},A_{1}+A_{2}}(f_{1},f_{2},J_{f}^{T})}
{\prod\limits_{\eta\in \Sigma} R_{A_{1}^{\eta},A_{2}^{\eta}}(f_{1}^{\eta},f_{2}^{\eta})^{\mu(\eta)}},
\end{equation}
and similarly for $\Delta_{A'_{1},A_{2}}(f'_{1},f_{2})$ and $\Delta_{A''_{1},A_{2}}(f''_{1},f_{2})$.
Let us write the numerator of~(\ref{deltafrac}) as follows:
\[
R_{A'_{1}+A''_{1},A_{2},A'_{1}+A''_{1}+A_{2}}
	(f'_{1}f''_{1},f_{2},J_{f'_{1}f''_{1},f_{2}}^{T}), 
\]
where $J^{T}_{f_{1}'f_{1}'',f_{2}} =
f_{1}'J^{T}_{f_{1}''f_{2}}+f_{1}''J^{T}_{f_{1}',f_{2}}$. 
Let us apply Theorem~\ref{th:productres} to re-write it as follows:
\[
R_{A'_{1},A_{2},A'_{1}+A''_{1}+A_{2}}
 (f'_{1},f_{2}, J^{T}_{f_{1}'f_{1}'',f_{2}} ) \,
R_{A''_{1},A_{2},A'_{1}+A''_{1}+ A_{2}}
 (f''_{1},f_{2}, J^{T}_{f_{1}'f_{1}'',f_{2}} )
=\]
\[=
R_{A'_{1},A_{2},A'_{1}+A''_{1}+A_{2}}
 (f'_{1},f_{2},f''_{1}J_{f'_{1},f_{2}}^{T}) \,
R_{A''_{1},A_{2},A'_{1}+A''_{1}+A_{2}}
 (f''_{1},f_{2},f'_{1}J_{f''_{1},f_{2}}^{T}),
\]
because the resultant of $\{h_1,h_2+ g h_1,\dots\}$ equals
the resultant of $\{h_1,h_2,\dots\}$, for any choice of
polynomials $h_1, h_2, g$ (with suitable supports).
We employ again Theorem~\ref{th:productres} to finalize
the numerator as follows: 
\[
R_{A'_{1},A_{2},A'_{1}+A_{2}}(f'_{1},f_{2},J_{f'_{1},f_{2}}^{T})
\cdot
R_{A''_{1},A_{2},A''_{1}+A_{2}}(f''_{1},f_{2},J_{f''_{1},f_{2}}^{T})
\cdot
R_{A^{'}_{1},A^{''}_{1},A_{2}}(f^{'}_{1},f^{''}_{1},f_{2})^{2}.
\]

For the denominator of~(\ref{deltafrac}), we
use again Theorem~\ref{th:productres} to write:
\begin{align*}
&\prod\limits_{\eta\in\Sigma'}R_{{A'_{1}}^{ \eta},A_{2}^{\eta}}({f'_{1}}^{\eta},f_{2}^{\eta})^{{\mu'}(\eta)}\cdot
\prod\limits_{\eta\in\Sigma''} R_{{A''_{1}}^{\eta},A_{2}^{\eta}}({f''_{1}}^{\eta},f_{2}^{\eta})^{{\mu''}(\eta)} =&\\
&\prod\limits_{\eta\in \Sigma} R_{{A'_{1}}^{ \eta}+{A''_{1}}^{\eta},A_{2}^{\eta}}({f'}_{1}^{\eta}{f''}_{1}^{\eta},f_{2}^{\eta})^{\mu(\eta)}
\cdot E,&
\end{align*}
because the products
\[\prod\limits_{\eta\in\Sigma \setminus \Sigma'}R_{{A'_{1}}^{ \eta},A_{2}^{\eta}}({f'_{1}}^{\eta},f_{2}^{\eta})^{{\mu'}(\eta)}=
\prod\limits_{\eta\in\Sigma \setminus \Sigma''} R_{{A''_{1}}^{\eta},A_{2}^{\eta}}({f''_{1}}^{\eta},f_{2}^{\eta})^{{\mu''}(\eta)} = 1, \]
since ${f'_{1}}^{\eta},f_{2}^{\eta}$ (resp. ${f''_{1}}^{\eta},f_{2}^{\eta}$) are both monomials.
To conclude the proof, simply assemble the above equations.
\end{proof}

As a consequence, we have $\deg_{A_{1},A_{2}}\Delta(f_{1},f_{2})=$
\[
=\deg_{A'_{1},A_{2}}\Delta(f'_{1},f_{2})+\deg_{A''_{1},A_{2}}
\Delta(f''_{1},f_{2})+2\cdot \deg_{A'_{1},A''_{1},A_{2}} R(f'_{1},f''_{1},f_{2}) - \deg(E).
\]

When all the configurations are full and with the same normal fan,
all the exponents ${\mu}(\eta) ={\mu'}(\eta)= {\mu''}(\eta)= 1$. 
Therefore,  $E=1$  and no extra factor occurs.

We define $\mu'_1(\eta), \mu''_1(\eta)$ as in \eqref{eq:mui}. Indeed, we now fix
$\eta$ and will simply write $\mu'_1, \mu''_1, \mu_1, \mu_2$.
It happens that only one of the factors associated to $\eta$ 
can occur in $E$ with non zero coefficient.
More explicitly, we have the following corollary, whose proof is straightforward.

\begin{cor}\label{cor:f}
With the notations of Corollary \ref{cor:factor}, 
for any $\eta \in \Sigma$ it holds that:
\begin{itemize}
\item If $\mu'_1= \mu''_1$, then $\mu' = \mu''=\mu$ and 
there is no factor in $E$ ``coming from $\eta$''.
\item  If $\mu'_1\neq \mu''_1$, assume wlog that $\mu_1 =\mu'_1 < \mu''_1$. 
There are three subcases:
\begin{itemize}
\item
If $\mu_2 \le  \mu_1$, again there is no factor in $E$ ``coming from $\eta$''.
\item
If $\mu_1 =\mu'_1 < \mu_2 < \mu''_1$, then the resultant
$\Res_{(A'_1)^\eta, A_2^\eta}((f'_1)^\eta, f_2^\eta)$ does not occur,
but $\Res_{(A''_1)^\eta, A_2^\eta}((f''_1)^\eta, f_2^\eta)$
has nonzero exponent (this resultant could just be the coefficient
of a vertex raised to the mixed multiplicity).
\item
If $\mu_1 =\mu'_1 < \mu''_1 \le \mu_2$, the situation is just the 
opposite than in the previous case.
\end{itemize}
\end{itemize}
\end{cor}

\section{Conclusion and future work}

The intent of this book chapter was to present our main results relating the
mixed discriminant with the sparse resultant of two bivariate Laurent 
polynomials with fixed support and their {toric Jacobian}. 
On our way, we deduced a general multiplicativity formula for the mixed discriminant 
when one polynomial factors as $f=f'\cdot f''$. 
This formula occurred as a consequence of our main result, Theorem \ref{th:main}, and
generalized known formulas in the homogeneous case to the sparse setting. 
Furthermore, we obtained a new proof of the bidegree formula
for planar mixed discriminants, which appeared in ~\cite{MD}. 

The generalization of our formulas to any number of variables will allow
us to extend our applications and to develop effective computational techniques
for sparse discriminants based on well tuned software for the computation of
resultants.

\paragraph{\small{\textsc{Acknowledgments}.}}
%\begin{acknowledgement}
A.~Dickenstein is partially supported by UBACYT
  20020100100242, CONICET PIP 112-200801-00483 and ANPCyT 2008-0902
  (Argentina). She also acknowledges the support of the M.~Curie
Initial Training Network ``SAGA" that made possible her visit 
to the University of Athens in January 2012, where this work was initiated.  
I.Z.\ Emiris was partially supported by the M.~Curie
Initial Training Network ``SAGA" (ShApes, Geometry,
Algebra), FP7-PEOPLE contract PITN-GA-2008-214584.
A.~Karasoulou's research has
received funding from the European Union (European Social Fund) 
and Greek national funds through the Operational Program 
``Education and Lifelong Learning" of the National Strategic 
Reference Framework, Research Funding Program ``ARISTEIA",
Project ESPRESSO: Exploiting Structure in Polynomial Equation and System
Solving with Applications in Geometric and Game Modeling.
%\end{acknowledgement}


\begin{thebibliography}{100}

\bibitem{OB} O.\ Benoist. Degr\'es d' homog\'en\'eit\'e de l' ensemble des intersactions compl\`etes singuli\`eres. \emph{Annales de l'Institut Fourier}, 62(3):1189--1214, 2012.

\bibitem{BJ} L.\ Bus\'e and J.-P.\ Jouanolou. A Computational 
approach to the Discriminant of homogeneous polynomials. Preprint, 
arXiv:1210.4697, 2012.

\bibitem{MD}
E.\ Cattani, M.A.\ Cueto, A.\ Dickenstein, S.\ Di Rocco and B.\ Sturmfels. 
Mixed Discriminants. To appear: Math.\ Z., 2013.

\bibitem{CDS} E.\ Cattani, A.\ Dickenstein, and B.\ Sturmfels.
Residues and resultants. 
\emph{Journal of Math.\ Sciences}, 5:119--148, University of Tokyo, 1998.

\bibitem{DS} C.\ D'Andrea and M.\ Sombra. Sparse resultants via 
multiprojective elimination. Preprint, 2012.
Available at: http://atlas.mat.ub.es/personals/sombra/publications.html.

\bibitem{D} A.\ Dickenstein. A world of binomials. 
\emph{Foundations of Computational Mathematics}, Hong Kong 2008, 
eds.\ F.\ Cucker, A.\ Pinkus, M.\ Todd, London Mathematical 
Society Lecture Note Series, 363:42--66, 2009.    

\bibitem{D13} A.\ Dickenstein. Mixed discriminants and toric jacobians.
Manuscript, 2013.

\bibitem{DRRS} A.\ Dickenstein, J.M.\ Rojas, K.\ Rusek and J.\ Shih. Extremal
real algebraic geometry and A-discriminants.\emph{ Moscow Mathematical
J.}, 7(3):425--452,  2007.

\bibitem{DT} A.\ Dickenstein and L.F.\ Tabera. Singular tropical hypersurfaces.   
\emph{Discrete \& Computational Geometry}, 47(2):430--453, 2012.

\bibitem{DFS} A.\ Dickenstein, E.M.\ Feichtner and B.\ Sturmfels. Tropical Discriminants.
\emph{J.\ Amer.\ Math.\ Soc.}, 20:1111--1133, 2007,

\bibitem{EFKP} I.Z.\ Emiris, V.\ Fisikopoulos, C.\ Konaxis, and L.\ Penaranda, 
An output-sensitive algorithm for 
computing projections of resultant polytopes. 
In \emph{Proc.\ ACM Symp.\ on Computational Geometry}, pages 179--188, 2012.

\bibitem{EKKL} I.Z.\ Emiris, T.\ Kalinka, C.\ Konaxis, 
and T.\ Luu Ba. Sparse implicitization by interpolation: 
Characterizing non-exactness and an application to computing discriminants. 
\emph{J.\ Computer Aided Design}, Special Issue, 45, pages 252--261, 2013.

\bibitem{ETT08} I.Z.\ Emiris, E.\ Tsigaridas, G.\ Tzoumas. 
The Predicates for the Exact Voronoi Diagram of 
Ellipses under the Euclidean Metric. \emph{Int.\ J.\ Comput.\ Geometry Appl.},
18(6):567--597, 2008.

\bibitem{E10}
A.\ Esterov. Newton polyhedra of discriminants of projections. 
\emph{Discrete \& Computational Geometry}, 44(1):96--148, 2010. 

\bibitem{E} A.\ Esterov. Discriminant of system of equations. 
Preprint, arXiv:1110.4060, 2011.  

\bibitem{FMRS} J.-C.\ Faug\'{e}re, G.\ Moroz, F.\ Rouillier,
and M.\ Safey El Din.
Classification of the Perspective-Three-Point problem:
Discriminant variety and real solving polynomial systems of inequalities.
In \emph{Proc.\ ACM ISSAC}, Hagenberg (Austria), pages 79--86, 2008.

\bibitem{gkz} I.M.\ Gel'fand, M.M.\ Kapranov and A.V.\ Zelevinsky, 
Discriminants, resultants and multidimensional 
determinants. Birkh\"auser, Boston, 1994.

\bibitem{gkz89} I.M.\ Gel'fand, M.M.\ Kapranov and A.V.\ Zelevinsky.
Hyper\-geo\-metric functions and toric varieties.
\emph{Funktsional.\ Anal.\ i Prilozhen.}, 23, No. 2:12--26, 1989.

\bibitem{GN} L.\ Gonz\'{a}lez-Vega and I.\ Nacula. 
Efficient topology determination of implicitly defined algebraic plane curves. 
\emph{Computer Aided Geometric Design}, 19(9):719--743, 2002.

\bibitem{MRCW} G.\ Moroz, F.\ Rouiller, D.\ Chablat and P.\ Wenger. 
On the determination of cusp points of 3-RPR parallel manipulators. 
\emph{Mechanism and Machine Theory}, 45(11):1555--1567, 2010.

\bibitem{N} J.\ Nie. Discriminants and non-negative polynomials. 
\emph{J.\ Symbolic Comput.}, 47:167--191, 2012.

\bibitem{PS} P.\ Pedersen and B.\ Sturmfels. Product formulas 
for resultants and Chow forms. \emph{Math.\ Z.}, 214:377--396, 1993.

\bibitem{Rin}\label{Rincon} E.F.\ Rinc\'on. Computing tropical linear spaces. 
\emph{J.\ Symbolic Comput.}, 51:86--93, 2013.

\bibitem{SS} M.\ Shub and S.\ Smale. Complexity of B\'ezout's theorem 
I.\ Geometric aspects. \emph{J.\ Amer.\ Math.\ Soc.}, 6(2):459--501, 1993.

\bibitem{S} B.\ Sturmfels. The Newton polytope of the resultant. 
\emph{J.\ Algebraic Combin.}, 3:207--236, 1994.

\bibitem{Syl} J.J.\ Sylvester. Sur l'extension de la th\'eorie 
des r\'esultants alg\'ebriques. \emph{Comptes Rendus de
l'Acad\'emie des Sciences}, LVIII:1074--1079, 1864.

\end{thebibliography}
\end{document}